\documentclass{article}
\usepackage{amsthm}
\usepackage{amsmath,amssymb}
\usepackage{helvet}
\usepackage{docmute}
\usepackage{tcolorbox}
\usepackage[enableskew]{youngtab}
\usepackage{color}
\usepackage{tikz}
\usepackage{ytableau}
\usepackage{braket}
\usepackage{autobreak}
\usepackage{mathtools}
\usepackage{hyperref}
\usepackage{endnotes}
\usepackage{authblk}
\hypersetup{%
 setpagesize=false,%
 bookmarks=true,%
 bookmarksdepth=tocdepth,%
 bookmarksnumbered=true,%
 colorlinks=false,%
 pdftitle={},%
 pdfsubject={},%
 pdfauthor={},%
 pdfkeywords={}}
\theoremstyle{definition}
\newtheorem{dfn}{Definition}[section]
\newtheorem{thm}{Theorem}[section]
\newtheorem{lem}{Lemma}[section]

\newtheorem{cor}{Corollary}[section]
\newtheorem*{main*}{Main result}
\theoremstyle{definition}
\newtheorem{case}{Example}[section]
\newtheorem*{prop*}{proposition}
\numberwithin{equation}{section}

\newtheorem{rem}{Remark}[section]
\theoremstyle{plain}

\allowdisplaybreaks

\makeatletter
\renewcommand{\AB@affilsep}{\quad\protect\Affilfont}
\let\AB@affilsepx\AB@affilsep 

\makeatother

\begin{document}

\title{A recursive method for the oddness of the number of set-valued tableaux}
\author{Taikei Fujii, Takahiko Nobukawa and Tatsushi Shimazaki}

\date{}

\maketitle

\begin{abstract}
Set-valued tableaux, introduced by Buch to express the tableaux-sum formula for stable Grothendieck polynomials, generalize semistandard tableaux.
We provide a new recursive proof that the number of set-valued tableaux of a given shape is odd.
\end{abstract}

\vspace{1mm}

\\

\section{Introduction}\label{I}
Let $\lambda$ be a partition of a non-negative integer, and identify it with its Young diagram.
The Schur polynomial $s_\lambda$ is well known to have a representation via semistandard Young tableaux.
Throughout this paper, we use the term semistandard tableaux as an abbreviation for semistandard Young tableaux. 
We denote by ${\rm SST}(\lambda,n)$ the set of semistandard tableaux of shape $\lambda$ with entries from $\{1,2,\dots,n\}$.
The number of such tableaux is given by the classical hook-length formula~\cite{[Mac95],[Nou23]}:
\begin{align}
\label{h}
|{\rm SST}(\lambda,n)|= \prod_{(i,j) \in \lambda}\frac{n+j-i}{h_{i,j}},
\end{align}
where $h_{i,j}$ is the hook-length of the element $(i,j) \in \lambda$. 
Thus, $|{\rm SST}(\lambda,n)|$ admits a factorized form and is readily computable.

The set-valued semistandard tableaux, introduced by Buch~\cite{[Buc02]}, generalize semistandard tableaux and play a central role in expressing the tableaux-sum formula for the (stable) Grothendieck polynomial $G_\lambda$.
We denote by ${\rm SVT}(\lambda,n)$ the set of set-valued semistandard tableaux of shape $\lambda$, where each box is filled with a non-empty subset of $\{1,2,\dots,n\}$  satisfying certain semistandard conditions (see Section~\ref{S2}).

The Grothendieck polynomial was introduced by Lascoux and Schützenberger~\cite{[Las90],[LS82]} to serve as representatives for the structure sheaves associated with Schubert varieties in flag varieties.
This geometric interpretation highlights their importance in algebraic combinatorics and $K$-theory.

We count the number of set-valued semistandard tableaux of shape $\lambda$ with entries at most $n$, denoted $|{\rm SVT}(\lambda,n)|$.
The following data from explicit computation suggest that, in contrast to the semistandard tableaux case, $|{\rm SVT}(\lambda,n)|$ admits no simple factorized expression.
For example, we have
\begin{align*}
&|{\rm SVT}((2,1),3)|=27,\quad |{\rm SVT}((2,2),3)|=13,\quad |{\rm SVT}((4,3),3)|=103,\\
&|{\rm SVT}((2,1),4)|=159,\quad |{\rm SVT}((2,2),4)|=97,\quad |{\rm SVT}((4,3),4)|=1759.
\end{align*}
The appearance of prime values in several cases further indicates that a factorized expression, like the hook-length formula, is unlikely to hold.
The enumeration problem for set-valued semistandard tableaux thus presents a nontrivial and intriguing challenge.
We provide an explicit formula for $|{\rm SVT}(\lambda,n)|$ in~\cite{FNS24}, derived from the bi-alternant formula for the Grothendieck polynomial $G_\lambda$ discussed in Appendix~\ref{A.1}.

An unexpected phenomenon also emerges: the total number $|{\rm SVT}(\lambda,n)|$ is always odd.
We prove that this parity holds in general, showing that the number of set-valued semistandard tableaux is odd for any skew shape.

\begin{main*}[Theorem~\ref{M}]
Let $\lambda$ and $\mu$ be Young diagrams such that $\lambda \supset \mu$.
For any skew Young diagram $\theta = \lambda/\mu$, the number $|{\rm SVT}(\theta,n)|$ is odd.
\end{main*}

The organization of this paper is as follows. 
In Section~\ref{S2}, we provide the necessary definitions. 
In Section~\ref{S3}, we prove the main result.
We summarize the work in Section~\ref{S4}.
Appendices~\ref{A.1},~\ref{B.1} and~\ref{C.1} are devoted to alternative proofs of the oddness property in the non-skew and skew cases, respectively.

The proof of the main theorem in Section~\ref{S3} focuses on the construction of ${\rm SVT}(\theta, n)$, and hence may be applied to counting the number of such tableaux.
The proofs in Appendices~\ref{A.1},~\ref{B.1} and~\ref{C.1} are more concise, as they aim only to obtain the main result.

\section{Preliminaries}\label{S2}
In this section, we define the main combinatorial objects in this paper.

\subsection{Partitions and Young diagrams}
A \emph{partition} of a non-negative integer $l$ is a finite sequence of non-negative integers
\begin{align*}
\lambda = (\lambda_1, \lambda_2, \dots, \lambda_r)
\end{align*}
such that $\lambda_1 \ge \lambda_2 \ge \dots \ge \lambda_r$ and $\lambda_1 + \lambda_2 + \dots + \lambda_r = l$.
We denote this by $\lambda \vdash l$. 
We identify $(\lambda_1,\dots,\lambda_s)$ with $(\lambda_1,\dots,\lambda_s,0)\ (s \in \mathbb{Z}_{>0})$.
The \emph{length} of $\lambda$, denoted $\ell(\lambda)$, is defined as the smallest integer $s$ such that $\lambda_s > 0$ and $\lambda_{s+1} = 0$.

The \emph{Young diagram of shape $\lambda$} is the set
\begin{align*}
\{ (i,j) \in \mathbb{Z}_{>0}^2 \mid 1 \leq i \leq \ell(\lambda),\ 1 \leq j \leq \lambda_i \}.
\end{align*}
We call an element $(i,j)$ of $\lambda$ a box. 
We write $|\lambda| = \sum_i \lambda_i$ for the number of boxes in the Young diagram.
The Young diagram consists of $|\lambda|$ boxes arranged in $\ell(\lambda)$ left-justified rows, where the $i$-th row contains $\lambda_i$ boxes. 
We follow the English convention, identifying $(i,j)$ with the box in the $i$-th row and $j$-th column, where $i$ increases downward and $j$ increases to the right. 
A partition and its Young diagram are regarded interchangeably.

\subsection{Semistandard tableaux}
Let $\lambda \vdash l$ and $[n] \coloneqq \{1, 2, \dots, n\}$. A \emph{semistandard tableau} of shape $\lambda$ is a filling of the boxes of the Young diagram of shape $\lambda$ with elements of $[n]$, such that:
\begin{itemize}
  \item entries weakly increase from left to right along each row,
  \item entries strictly increase from top to bottom along each column.
\end{itemize}
We denote the set of such tableaux by ${\rm SST}(\lambda)$, and by ${\rm SST}(\lambda,n)$ when emphasizing the entries are from $[n]$.
It is well known that the Schur polynomial $s_\lambda$ has a tableaux-sum expression over ${\rm SST}(\lambda,n)$.

\subsection{Set-valued tableaux}
Let $b_{i,j}$ denote the box at position $(i,j)$ in the Young diagram $\lambda$, and let $T_{i,j}$ be a non-empty subset of $[n]$.
\begin{dfn}[\cite{[Buc02]}]\label{svt}
For a Young diagram $\lambda$, a \emph{set-valued semistandard tableau of shape $\lambda$} is a filling of each box $b_{i,j}$ with a non-empty subset $T_{i,j} \subset [n]$ that satisfies the following conditions:
\begin{itemize}
  \item $\max T_{i,j} \leq \min T_{i,j+1}$,
  \item $\max T_{i,j} < \min T_{i+1,j}$.
\end{itemize}
\end{dfn}
We term these conditions the \textit{semistandard condition}.
We refer to set-valued semistandard tableaux simply as set-valued tableaux. 
We denote the set of set-valued tableaux of shape $\lambda$ with entries in $[n]$ by ${\rm SVT}(\lambda,n)$. The case ${\rm SVT}(\lambda,n) = \emptyset$ is excluded from consideration in this paper.
Set-valued tableaux generalize semistandard tableaux. 
Buch introduced them to express the tableaux-sum formula for the Grothendieck polynomial $G_\lambda$~\cite{[Buc02]}. Since $G_\lambda$ is not used directly in our proof of the main result, we defer its discussion to Appendix~\ref{A.1}.

\begin{case}\label{ex2.1}
Let $\lambda = (2,1)={\scalebox{0.4}{{\raisebox{8pt}[0pt][0pt]{\ytableaushort{\ \ ,\ }}}}}$ and $n = 3$.
An example $T \in {\rm SVT}({\scalebox{0.3}{{\raisebox{10pt}[0pt][0pt]{\ytableaushort{\ \ ,\ }}}}},3)$ is given by:\vspace{-3mm}
\begin{align*}
T = \quad {\raisebox{2pt}{\ytableaushort{{1}{1,\!2},{2,\!3}}}}.
\end{align*}
The subset $T_{2,1} = \{2,3\} \subset [3]$ is assigned to the box $b_{2,1}$ in this example. To simplify notation, we abbreviate $\scriptsize{{\raisebox{-4.0pt}[0pt][0pt]{\ytableaushort{{2,\!3}}}}}$ as $\scriptsize{{\raisebox{-4.0pt}[0pt][0pt]{\ytableaushort{{23}}}}}$.
All tableaux in ${\rm SVT}({\scalebox{0.3}{{\raisebox{10pt}[0pt][0pt]{\ytableaushort{\ \ ,\ }}}}},3)$ are enumerated below:\vspace{-2mm}
\begin{align*}
&{\raisebox{-10pt}[0pt][0pt]{\ytableaushort{11,2}}}{\raisebox{-7pt}[0pt][0pt],\ } {\raisebox{-10pt}[0pt][0pt]{\ytableaushort{11,3}}}{\raisebox{-7pt}[0pt][0pt],\ }{\raisebox{-10pt}[0pt][0pt]{\ytableaushort{12,2}}}{\raisebox{-7pt}[0pt][0pt],\ }{\raisebox{-10pt}[0pt][0pt]{\ytableaushort{12,3}}}{\raisebox{-7pt}[0pt][0pt],\ }{\raisebox{-10pt}[0pt][0pt]{\ytableaushort{13,2}}}{\raisebox{-7pt}[0pt][0pt],\ }{\raisebox{-10pt}[0pt][0pt]{\ytableaushort{13,3}}}{\raisebox{-7pt}[0pt][0pt],\ }{\raisebox{-10pt}[0pt][0pt]{\ytableaushort{22,3}}}{\raisebox{-7pt}[0pt][0pt],\ }{\raisebox{-10pt}[0pt][0pt]{\ytableaushort{23,3}}}{\raisebox{-7pt}[0pt][0pt],\ }\\\\\\
&{\raisebox{-3pt}[0pt][0pt]{\ytableaushort{1{12},2}},\ }{\raisebox{-3pt}[0pt][0pt]{\ytableaushort{1{13},2}},\ }{\raisebox{-3pt}[0pt][0pt]{\ytableaushort{1{23},2}},\ } {\raisebox{-3pt}[0pt][0pt]{\ytableaushort{1{12},{3}}},\ }{\raisebox{-3pt}[0pt][0pt]{\ytableaushort{1{13},3}},\ }{\raisebox{-3pt}[0pt][0pt]{\ytableaushort{1{23},3}},\ }{\raisebox{-3pt}[0pt][0pt]{\ytableaushort{1{1},{23}}},\ }{\raisebox{-3pt}[0pt][0pt]{\ytableaushort{1{2},{23}}},\ }\\\\
&{\raisebox{-10pt}[0pt][0pt]{\ytableaushort{1{3},{23}}}}{\raisebox{-7pt}[0pt][0pt],\ }{\raisebox{-10pt}[0pt][0pt]{\ytableaushort{2{23},{3}}}}{\raisebox{-7pt}[0pt][0pt],\ }{\raisebox{-10pt}[0pt][0pt]{\ytableaushort{{12}2,3}}}{\raisebox{-7pt}[0pt][0pt],\ }{\raisebox{-10pt}[0pt][0pt]{\ytableaushort{{12}3,3}}}{\raisebox{-7pt}[0pt][0pt],\ }{\raisebox{-10pt}[0pt][0pt]{\ytableaushort{{1}{12},{23}}}}{\raisebox{-7pt}[0pt][0pt],\ }{\raisebox{-10pt}[0pt][0pt]{\ytableaushort{1{13},{23}}}}{\raisebox{-7pt}[0pt][0pt],\ }{\raisebox{-10pt}[0pt][0pt]{\ytableaushort{1{23},{23}}}}{\raisebox{-7pt}[0pt][0pt],\ }{\raisebox{-10pt}[0pt][0pt]{\ytableaushort{{12}{23},3}}}{\raisebox{-7pt}[0pt][0pt],\ }\\\\
&{\raisebox{-17pt}[0pt][0pt]{\ytableaushort{1{123},{2}}}}{\raisebox{-14pt}[0pt][0pt],\ }{\raisebox{-17pt}[0pt][0pt]{\ytableaushort{1{123},{3}}}}{\raisebox{-14pt}[0pt][0pt],\ }{\raisebox{-17pt}[0pt][0pt]{\ytableaushort{1{123},{23}}}}{\raisebox{-14pt}[0pt][0pt].\ }\\\\\\
\end{align*}
\end{case}

\subsection{Skew Young diagrams and set-valued skew tableaux}
Let $\lambda$ and $\mu$ be Young diagrams with $\lambda \supset \mu$. The set-theoretic difference $\theta = \lambda - \mu$ is called a \emph{skew Young diagram}, denoted $\theta = \lambda/\mu$. It consists of the boxes in $\lambda$ that are not in $\mu$. We define $|\theta|$ to be the number of boxes in $\theta$.

A \emph{set-valued skew tableau} of shape $\theta$ is a filling of each box in $\theta$ with a non-empty subset of $[n]$, satisfying the semistandard condition in Definition~\ref{svt}. We denote the set of such tableaux by ${\rm SVT}(\theta,n)$.
For ${\rm SVT}(\theta,n)$ to be non-empty, each column of $\theta$ must have at most $n$ boxes.
This is because if a column has more than $n$ boxes, the elements within that column must be strictly increasing across boxes, implying more than $n$ distinct elements are needed from $[n]$, which is impossible.
We consider only the case $\mathrm{SVT}(\theta,n) \neq \emptyset$.

\begin{case}
Let $\lambda = (5,3,2,1)$, $\mu = (3,2)$, and $n = 3$. Thus, we have
\begin{align*}
\theta = \{ b_{3,1}, b_{4,1}, b_{3,2}, b_{2,3}, b_{1,4}, b_{1,5} \},\quad |\theta| = 6.
\end{align*}
An example $T \in {\rm SVT}(\theta,3)$ is as follows:
\begin{align*}
T = \quad
\begin{ytableau}
\none & \none & \none & {1} & {123} \\
\none & \none & {23} \\
{1} & {13} \\
{2}
\end{ytableau}.
\end{align*}
Note that ${\rm SVT}(\theta,3)$ is non-empty even though $\ell(\lambda) > 3$. 
\end{case}

\section{The odd property for \(|\mathrm{SVT}(\theta,n)|\)}\label{S3}

In this section, we show the parity property for the number of set-valued tableaux of any skew Young diagram.

The argument proceeds by induction on the number of boxes $|\theta|$.
For the inductive step, we remove the rightmost box in the lowest non-empty row
and analyze how tableaux on the smaller shape extend to those on $\theta$.
We establish four lemmas to prove Theorem~\ref{M}.

Let $b_{i,j}$ denote the rightmost box in the bottommost row of $\theta$; that is,
\begin{align*}
&i=\max\{x\mid b_{x,y}\in\theta\},\\
&j=\max\{y\mid b_{i,y}\in\theta\}.
\end{align*}
We denote $\theta'=\theta\setminus\{b_{i,j}\}$.
For each tableau $T'\in\mathrm{SVT}(\theta',n)$,
define
\begin{align*}
U(T')=\{\,S\subset[n]\mid S\neq\emptyset,\ (T',S)\in\mathrm{SVT}(\theta,n)\,\}.
\end{align*}
Namely, $U(T')$ is the set of all non-empty subsets $S \subset [n]$ that can be assigned to the box $b_{i,j}$.
The resulting tableau $T = (T', S)$ is required to satisfy the semistandard condition and have the shape $\theta$.
This assignment must specifically preserve the semistandard condition with respect to the adjacent boxes $b_{i,j-1}$ and $b_{i-1,j}$, if they exist.

\begin{lem}\label{lem:decomp}
Let $\theta$ be a skew Young diagram.
The cardinality of the set of set-valued tableaux of shape $\theta$ satisfies the following:
\begin{align}\label{dec}
|{\rm SVT}(\theta,n)|=\sum_{T'\in {\rm SVT}(\theta',n)}|U(T')|.
\end{align}
Furthermore, for any $T' \in \mathrm{SVT}(\theta',n)$, its cardinality $|U(T')|$ is odd if the set $U(T')$ is non-empty.
\end{lem}
\begin{proof}
The skew Young diagram $\theta$ can be uniquely decomposed into the skew Young diagram $\theta' = \theta \setminus {b_{i,j}}$ and the box ${b_{i,j}}$.
This decomposition implies that a set-valued tableau $T \in \mathrm{SVT}(\theta,n)$ is uniquely determined by the following tableau $T'$ and set $S$:
\begin{enumerate}
\item The restriction $T' = T|_{\theta'}$ of $T$ to the shape $\theta'$, which belongs to $\mathrm{SVT}(\theta',n)$ because each $T_{i',j'}$ for $b_{i',j'}$ satisfies the semistandard condition.
\item The set $S = T_{i,j}$ that is assigned to the box $b_{i,j}$.
\end{enumerate}
For a set-valued tableau $T$ to belong to $\mathrm{SVT}(\theta,n)$, the set $S$ must satisfy the semistandard condition with respect to $T'$.
The set of all possible assignments $S$ for a fixed $T'$ is exactly the set $U(T')$.
Since every $T \in \mathrm{SVT}(\theta,n)$ is uniquely decomposed by choosing a pair $(T', S)$, where $T' \in \mathrm{SVT}(\theta',n)$ and $S \in U(T')$, the total number of tableaux is the sum of the number of choices for $S$ over all possible $T'$.
Thus, we have the decomposition equality~\eqref{dec}.

A set $S$ can be assigned to $b_{i,j}$ if and only if the semistandard condition is satisfied with respect to the adjacent boxes in $T'$. 
That is, $S$ must be a non-empty subset of $[n]$ satisfying:
\begin{align}
\begin{cases}
\max T'_{i,j-1} \le&\!\!\!\! \min S, \\
\max T'_{i-1,j} <&\!\!\!\! \min S.
\end{cases}\label{eq:lower-bounds}
\end{align}
These two conditions can be combined into a single lower bound for $\min S$. 
We define $a_{T'}$ by
\begin{align*}
a_{T'}=\max\bigl(\{0\}\cup\{\max T'_{i,j-1}\}\cup\{\max T'_{i-1,j}+1\}\bigr).
\end{align*}
Note that the maximum is taken only on the existing neighboring boxes of $b_{i,j}$ in $\theta'$.
From ~\eqref{eq:lower-bounds} and the definition of $a_{T'}$, we must have $\min S \ge a_{T'}$.
Therefore, $U(T')$ is precisely the set of all non-empty subsets $S$ of the set $A_{T'} = \{ a_{T'}, a_{T'}+1, \dots, n \}$.
If $U(T') \neq \emptyset$, it follows that $|A_{T'}| = n-a_{T'} +1 \ge 1$. 
Thus, we have
\begin{align*}
|U(T')|= 2^{\,n-a_{T'}+1}-1.
\end{align*}
Therefore, the number $|U(T')|$ is odd if $U(T') \neq \emptyset$.
This finishes the proof.
\end{proof}

For each $T'\in\mathrm{SVT}(\theta',n)$, we define an integer $h_{T'}\in\{0,1,\dots,n-1\}$ as the largest non-negative integer $h$ such that the entry $n-t+1$ occurs in the set $T'_{i-t,j}$ (if the box $b_{i-t,j}$ exists) for every $1\le t\le h$.
Furthermore, if the box $b_{i-h-1,j}$ exists, the entry $n-h$ must not appear in $T'_{i-h-1,j}$.
This definition yields the following disjoint decomposition:
\begin{align*}
\mathrm{SVT}(\theta',n)
 =\bigsqcup_{h=0}^{n-1}\mathrm{SVT}(\theta',n)_h,
\end{align*}
where $\mathrm{SVT}(\theta',n)_h=\{T'\in\mathrm{SVT}(\theta',n)\mid h_{T'}=h\}$.

\begin{rem}
	By the definition of $h_{T'}$, we have
	\begin{align}\label{remeqsingle}
		T'_{i-t,j}=\{n-t+1\}\quad (1\leq t\leq h_{T'}-1),
	\end{align}
	and 
	\begin{align}\label{remeqmax}
		\max T'_{i-h_{T'},j} = n-h_{T'}+1,
	\end{align}
	for $T'\in\mathrm{SVT}(\theta',n)$.
	We verify these in the following.
	For $1\leq t\leq h_{T'}$, we have
	\begin{align*}
		\min T'_{i-t,j}\leq n-t+1\leq \max T'_{i-t,j}
	\end{align*}
	since $n-t+1\in T'_{i-t,j}$.
	Due to the semistandardness,  we get
	\begin{align*}
		n-t+1\leq \max T'_{i-t,j}<\min T'_{i-t+1,j}\leq n-t+2.
	\end{align*}
	Therefore we have
	\begin{align*}
		&\max T'_{i-t,j}=n-t+1\quad (1\leq t\leq h_{T'}),\\
		&\min T'_{i-t+1,j}=n-t+2\quad (1\leq t\leq h_{T'}-1).
	\end{align*}
	This completes the check of \eqref{remeqsingle} and \eqref{remeqmax}.
\end{rem}

\begin{lem}\label{lem:U_empty}
Let $\theta$ be a skew Young diagram. For each positive integer $h$ with $1\le h\le n-1$, the set $U(T')$ is the empty set for all $T'\in\mathrm{SVT}(\theta',n)_h$.
\end{lem}
\begin{proof}
We fix $h \ge 1$ and consider a tableau $T'\in\mathrm{SVT}(\theta',n)_h$.
Due to~\eqref{remeqsingle}, we have $n\in T'_{i-1,j}, n-1\in T'_{i-2,j}, \dots, n-h+1\in T'_{i-h,j}$.
Since $h \ge 1$, we obtain $b_{i-1,j}\in\theta'$ and $n\in T'_{i-1,j}$.
As $n$ is the largest possible entry, it follows that $\max T'_{i-1,j}=n$.
Hence, we require $\min S > \max T'_{i-1,j}$, which implies $\min S > n$.
This makes it impossible for any non-empty subset $S \subset [n]$ to satisfy the semistandard condition.
Therefore, $U(T')=\emptyset$ for all $T' \in \mathrm{SVT}(\theta',n)_h$ when $h \ge 1$.
This completes the proof.
\end{proof}

We define a map $f:\mathrm{SVT}(\theta',n)_h\to\mathrm{SVT}(\theta',n)_h$ that modifies only the set $T'_{i-h,j}$:
\begin{align*}
f(T')_{i-h,j}=
\begin{cases}
T'_{i-h,j}\setminus\{n-h\} & (n-h\in T'_{i-h,j}),\\
T'_{i-h,j}\sqcup\{n-h\} & (n-h\notin T'_{i-h,j}).
\end{cases}
\end{align*}
All other sets $T'_{i',j'}\ (b_{i',j'} \neq b_{i-h,j})$ remain unchanged by the map $f$. 

\begin{lem}\label{lem:even-part}
Let $\theta$ be a skew Young diagram. For each positive integer $h$ with $1\le h\le n-1$, the map $f$ is a well-defined involution such that $f(T') \ne T'$ for all $T' \in \mathrm{SVT}(\theta',n)_h$. 
In particular, the number $|\mathrm{SVT}(\theta',n)_h|$ is even for $1\leq h\leq n-1$.
\end{lem}
\begin{proof}
	We fix $h\geq 1$.
	First we show that the map $f$ is well-defined, i.e. $f(T')\in\mathrm{SVT}(\theta',n)_h$ for $T'\in\mathrm{SVT}(\theta',n)_h$.
	
	We consider the case $n-h\in T'_{i-h,j}$.
	By the definition of $h_{T'}$, every $T'\in\mathrm{SVT}(\theta',n)_h$ satisfies the            condition $n-h+1\in T'_{i-h,j}$.
	Thus we have $f(T')_{i-h,j}\neq \emptyset$.
	It follows from $\min f(T')_{i-h,j}=\min (T'_{i-h,j}\setminus\{n-h\})\geq \min T'_{i-h,j}$ and $\max f(T')_{i-h,j}=\max(T')_{i-h,j}$, the semistandard conditions
	\begin{itemize}
		\item $\max f(T')_{i-h,j-1}\leq \min f(T')_{i-h,j}$,\quad $\max f(T')_{i-h,j}\leq \min f(T')_{i-h,j+1}$,
		\item $\max f(T')_{i-h-1,j}<\min f(T')_{i-h,j}$,\quad $\max f(T')_{i-h,j}<\min f(T')_{i-h+1,j}$,
	\end{itemize}
	are verified.
	
	We consider the case $n-h\notin T'_{i-h,j}$.
	Due to the definition of $h_{T'}$,  we have $n-h+1\in T'_{i-h,j}$ and 
	\begin{align*}
		\max T'_{i-h-1,j}<n-h
	\end{align*}
	if $b_{i-h-1,j}\in\theta'$.
	On the other hand, we find
	\begin{align*}
		\max T'_{i-h-1,j}<\min T'_{i-h,j}
	\end{align*}
	due to the semistandardness of $T'$.
	Therefore we derive
	\begin{align}
		\max f(T')_{i-h-1,j}=\max T'_{i-h-1,j}<\min (T'_{i-h}\sqcup\{n-h\})=\min f(T')_{i-h,j}.
	\end{align}
	If $b_{i-h,j-1}\in\theta'$, we have
	\begin{align}\label{lemma3proofineq2}
		\max T'_{i-h,j-1}\leq \min T'_{i-h,j}\leq n-h,
	\end{align}
	by the semistandardness of $T'$ and \eqref{remeqmax}.
	Therefore we obtain
	\begin{align}
		\max f(T')_{i-h,j-1}\leq \min (T'_{i-h,j}\sqcup \{n-h\})=\min f(T')_{i-h,j}.
	\end{align}
	The inequalities
	\begin{align*}
		\max f(T')_{i-h,j}<\min f(T')_{i-h+1,j},\quad \max f(T')_{i-h,j}\leq \min f(T')_{i-h,j+1},
	\end{align*}
	are verified from the semistandardness of $T'$ and \eqref{remeqmax}, more precisely the equation
	\begin{align*}
		\max f(T')_{i-h,j}=\max (T'_{i-h,j}\sqcup \{n-h\})=n-h+1=\max T'_{i-h,j}
	\end{align*}
	holds.
	
	The semistandard conditions for other boxes are obvious for both cases.
	In conclusion, we have $f(T')\in \mathrm{SVT}(\theta',n)$.
	By the definitions of $h_{T'}$ and $f$, we easily find $h_{f(T')}=h_{T'}$ for $T'\in\mathrm{SVT}(\theta',n)$.
	This gives the well-definedness of $f$, i.e. $f(T')\in\mathrm{SVT}(\theta',n)_h$.
	
	 Next, we show that $f$ is an involution on $\mathrm{SVT}(\theta',n)_h$ such that $f(T')\neq T'$.
	 We have
	 \begin{align*}
	 	f(f(T'))_{i-h,j}
	 	&=
	 	\begin{cases}
	 		f(T'_{i-h,j}\setminus\{n-h\})_{i-h,j} & (n-h\in T'_{i-h,j}),\\
	 		f(T'_{i-h,j}\sqcup\{n-h\})_{i-h,j} & (n-h\notin T'_{i-h,j}),
	 	\end{cases}
 		\\
 		&=
 		\begin{cases}
 			(T'_{i-h,j}\setminus\{n-h\})\sqcup \{n-h\} & (n-h\in T'_{i-h,j}),\\
 			(T'_{i-h,j}\sqcup\{n-h\})\setminus \{n-h\} & (n-h\notin T'_{i-h,j}),
 		\end{cases}
 		\\
 		&=T'_{i-h,j}.
	 \end{align*}
 	This leads to $f(f(T'))=T'$ for $T'\in\mathrm{SVT}(\theta',n)_h$.
 	It is obvious by the definition of $f$ that $f(T')_{i-h,j}\neq T'_{i-h,j}$.
 	Therefore the map $f$ is an involution on $\mathrm{SVT}(\theta',n)_h$ such that $f(T')\neq T'$.
 	
 	The map $f$ forms pairs for $\mathrm{SVT}(\theta',n)_h$, i.e.
 	\begin{align*}
 		\mathrm{SVT}(\theta',n)_h=\{T_1',f(T_1')\}\sqcup\{T_2',f(T_2')\}\sqcup\cdots
 	\end{align*}
 	In particular, we find that $|\mathrm{SVT}(\theta',n)_h|$ is even.
\end{proof}
With the parity of $|\mathrm{SVT}(\theta',n)_h|$ for $h \ge 1$ established, we can now prove the congruence relation for the total count.
\begin{lem}\label{lem:h0-odd}
Let $\theta$ be a skew Young diagram. 
If the number $|\mathrm{SVT}(\theta',n)|$ is odd, then the number $|\mathrm{SVT}(\theta',n)_0|$ is also odd. 
Furthermore, we have
\begin{align*}
|\mathrm{SVT}(\theta,n)|\equiv|\mathrm{SVT}(\theta',n)_0|\pmod2.
\end{align*}
\end{lem}

\begin{proof}
The total set of set-valued tableaux is decomposed on the value $h_{T'}$. 
Note that the component $\mathrm{SVT}(\theta',n)_h$ consists of tableaux $T'$ such that $h_{T'}=h$.
Thus, we have
\begin{align*}
|\mathrm{SVT}(\theta',n)|
   =\sum_{h=0}^{n-1}|\mathrm{SVT}(\theta',n)_h|.
\end{align*}
By Lemma~\ref{lem:even-part}, every term $|\mathrm{SVT}(\theta',n)_h|$ for $h\ge1$ is even.
It follows that
\begin{align*}
	|\mathrm{SVT}(\theta',n)| = |\mathrm{SVT}(\theta',n)_0| + \sum_{h=1}^{n-1} |\mathrm{SVT}(\theta',n)_h| \equiv |\mathrm{SVT}(\theta',n)_0| \pmod 2,
\end{align*}
in particular, $|\mathrm{SVT}(\theta',n)_0|$ is odd if $|\mathrm{SVT}(\theta',n)|$ is odd.

Next, we establish the parity relation for $|\mathrm{SVT}(\theta,n)|$. Lemma~\ref{lem:decomp} states that the total number of set-valued tableaux is given by the sum over the components:
\begin{align*}
|\mathrm{SVT}(\theta,n)| = \sum_{T'\in\mathrm{SVT}(\theta',n)}|U(T')| = \sum_{h=0}^{n-1} \sum_{T'\in\mathrm{SVT}(\theta',n)_h}|U(T')|.
\end{align*}
By Lemma~\ref{lem:U_empty}, the condition $h_{T'} \ge 1$ implies that $U(T') = \emptyset$.
From this, it follows that the total sum reduces to the $h=0$ component:
\begin{align*}
|\mathrm{SVT}(\theta,n)|
 = \sum_{T'\in\mathrm{SVT}(\theta',n)_0}|U(T')|.
\end{align*}
Since $\mathrm{Lemma}~\ref{lem:decomp}$ ensures that $|U(T')|$ is odd for all $T'\in\mathrm{SVT}(\theta',n)_0$, reducing the expression modulo 2 yields:
\begin{align*}
|\mathrm{SVT}(\theta,n)| \equiv \sum_{T'\in\mathrm{SVT}(\theta',n)_0} 1 \equiv |\mathrm{SVT}(\theta',n)_0| \pmod2.
\end{align*}
The congruence holds, and the claim follows.
\end{proof}

Applying the results of the preceding lemmas, we proceed to the inductive proof of the main theorem.

\begin{thm}\label{M}
Let $\lambda$ and $\mu$ be Young diagrams such that $\lambda \supset \mu$.
For any skew Young diagram $\theta=\lambda/\mu$, the number $|{\rm SVT}(\theta,n)|$ is odd.
\end{thm}

\begin{proof}
We prove the statement by induction on the number of boxes in the skew shape, $|\theta|$.
If the skew Young diagram $\theta$ consists of a single box, $|\theta|=1$. The number of possible entries for this single box is the number of non-empty subsets of $[n]$, which is $2^n - 1$.
The base case holds.

Assume that the theorem holds for all skew shapes with fewer than $|\theta|$ boxes.
We take a skew Young diagram $\theta$ with $|\theta| > 1$ and consider $\theta' = \theta \setminus \{b_{i,j}\}$. 
By the induction hypothesis, $|\mathrm{SVT}(\theta',n)|$ is odd.
Thus, we have the following two parity results by Lemma~\ref{lem:h0-odd}:
\begin{enumerate}
    \item The number of tableaux with $h_{T'}=0$ must be odd:
\begin{align*}
|\mathrm{SVT}(\theta',n)_0| \equiv 1 \pmod 2.
\end{align*}
    \item The parity of the original number $|\mathrm{SVT}(\theta,n)|$ is equivalent to the parity of the number of tableaux with $h_{T'}=0$:
\begin{align*}
    |\mathrm{SVT}(\theta,n)| \equiv |\mathrm{SVT}(\theta',n)_0| \pmod 2.
\end{align*}
\end{enumerate}
Combining these results, we obtain
\begin{align*}
|\mathrm{SVT}(\theta,n)| \equiv |\mathrm{SVT}(\theta',n)_0| \equiv 1 \pmod 2.
\end{align*}
Hence, $|\mathrm{SVT}(\theta,n)|$ is odd.
This establishes the theorem by induction.
\end{proof}

\begin{rem}
The inductive proof reveals the recursive structure of set-valued tableaux. Specifically, a tableau on $\theta'$ extends to one on $\theta$ by assigning a subset to the removed box $b_{i,j}$. Furthermore, the involution in Lemma~\ref{lem:even-part} is crucial as it eliminates all even contributions, simplifying the parity count. For an alternative short argument employing a global sign-reversing involution, refer to Appendix~\ref{B.1}.
\end{rem}
\begin{case}\label{3.1}
Let $\theta=(3,2,2)/(2,1)$ and $n=3$.
The skew Young diagram $\theta$ is as follows:
\begin{align*}
{\raisebox{-5.5pt}[0pt][0pt]{$\theta$\ =\ }} {\raisebox{1pt}[0pt][0pt]{\ytableaushort{\none\none\ ,\none\ ,\ \ }}.\ }\\[3mm]
\end{align*}
\vspace{0mm}
The box $b_{i,j}$ (rightmost of the lowest row of $\theta$) is $b_{3,2}$.
Removing $b_{3,2}$ yields $\theta'$:
\vspace{0mm}
\begin{align*}
{\raisebox{-5.5pt}[0pt][0pt]{$\theta'$\ =\ }} {\raisebox{1pt}[0pt][0pt]{\ytableaushort{\none\none\ ,\none\ ,\ }}.}\\[3mm]
\end{align*}
Take a tableau $T' \in {\rm SVT}(\theta',3)$ as follows:
\begin{align*}
{\raisebox{-5.5pt}[0pt][0pt]{$T'$\ =\ }} {\raisebox{1pt}[0pt][0pt]
{\ytableaushort{\none\none{123},\none{1},{1}}}}.
\end{align*}
For this tableau $T'$, we compute $h_{T'}$. In this situation, $n = 3$ and the removed box $b_{i,j}$ is $b_{3,2}$.
Thus, we examine column $j=2$ and the cells above $b_{3,2}$. We check the cell $T'_{i-1,j} = T'_{2,2}$ (the cell immediately above $b_{3,2}$ in the second column). Since $n=3 \notin T'_{2,2}=\{1\}$, the condition for $h \ge 1$ is not met. 
Therefore, by definition, $h_{T'}=0$.
Hence, $U(T') = \{ \{2\}, \{3\}, \{2,3\} \}$.
The tableaux in ${\rm SVT}(\theta,3)$ corresponding to $T'$ are listed below:
\vspace{1.5mm}
\begin{align*}
{\raisebox{1pt}[0pt][0pt]{\ytableaushort{\none\none{123},\none{1},{1}{2}}},\ }{\raisebox{1pt}[0pt][0pt]{\ytableaushort{\none\none{123},\none{1},{1}{3}}},\ }{\raisebox{1pt}[0pt][0pt]{\ytableaushort{\none\none{123},\none{1},{1}{23}}}.\ }
\end{align*}
\end{case}

\vspace{5mm}
\section{Conclusions}\label{S4}

In this paper, we show the oddness of the number of set-valued tableaux in skew Young diagrams (Theorem~\ref{M}). This property is established by two distinct methods: the detailed induction presented in Section~\ref{S3} and the sign-reversing involution provided in Appendix~\ref{B.1}.
The inductive proof in Section~\ref{S3} is particularly important as it explicitly reveals the recursive nature of set-valued tableaux and illustrates how the parity arises from the local structure of each tableau. This approach offers a structural and constructive perspective that enhances our combinatorial understanding of the theorem.
On the other hand, the sign-reversing involution presented in Appendix~\ref{B.1} provides a concise and elegant alternative proof. While technically simpler, this argument reflects a refined application of a powerful combinatorial principle and offers a complementary perspective. It allows the main result to be understood from a broader angle, emphasizing its robustness and deeper connections to the theory of involutive bijections.
We also include in Appendix~\ref{C.1} a simple proof based on a map to semistandard tableaux obtained by taking the maximum entry in each box.
These different viewpoints may also serve as a foundation for further enumerative investigations.

For the case of non-skew Young diagrams, such a property can be proved by using some formulas for the Grothendieck polynomial (see Corollary~\ref{AC} in Appendix~\ref{A.1}).  
In \cite{FNS24}, we presented an explicit formula for $|{\rm SVT}(\lambda,n)|$, derived from the bi-alternant formula for the Grothendieck polynomial.  
It remains unclear whether the oddness of the number of set-valued semistandard tableaux, established in this paper, can be derived directly from the explicit formula we obtained.  
Since this formula is related to hypergeometric functions, it would be of considerable interest to further investigate the enumeration of such tableaux from the perspective of special function theory.

In the following, we remark about related things for the number of set-valued tableaux $|{\rm SVT}(\lambda,n)|$.

\begin{rem}
The odd property for the non-skew case of $\theta$ was essentially proved in~\cite{[IS14]}.
We consider $\theta=\lambda/\emptyset=\lambda$.
It follows from $|{\rm SVT}(\theta,n)^{\emptyset\mathchar`-{\rm good}}|=1$ and $|{\rm SVT}(\theta,n)^{\emptyset\mathchar`-{\rm bad}}|$ is even due to the involution $\tau : {\rm SVT}(\theta,n)^{\lambda\mathchar`-{\rm bad}}\rightarrow{\rm SVT}(\theta,n)^{\lambda\mathchar`-{\rm bad}}$.
For more details of ${\rm SVT}(\theta,n)^{\lambda \mathchar`-{\rm good}}$ and the involution $\tau$, see Subsection 2.3, Lemma 3 in~\cite{[IS14]}, respectively.
\end{rem}
\begin{rem}
In~\cite{[IN09]}, the {\it excited Young diagram} was introduced.
The representation of the factorial Schur $Q$- and $P$-functions by these diagrams was given.
By changing the weight in this representation suitably, the stable Grothendieck polynomial $G_\lambda$ can be also represented by the excited Young diagrams as follows~\cite{[Nar11]}:
\begin{align*}
G_\lambda=G_\lambda(x_1,\dots,x_n ; \beta) = \sum_{D \in \mathcal{E}_n(\lambda)}W_{t_\beta}(D),
\end{align*}
where
\begin{align*}
W_{t_\beta}(D)=x^{w(D)}\prod_{b \in \mathcal{B}}(1+\beta x^{w(b)}).
\end{align*}
We can easily find that all weights $W_{t_\beta}(D)$ are even by putting $x_1=\cdots=x_n=1$ and $\beta=1$, except for the only condition not excited $D \in \mathcal{E}_n(\lambda)$. 

On the other hand, according to the tableaux-sum formula of $G_\lambda$~\cite{[Buc02]} (see~(\ref{G})), we have
\begin{align}\label{SV}
|{\rm SVT}(\lambda,n)| = G_\lambda(1,1,\dots,1 ; 1).
\end{align}
Thus, $|{\rm SVT}(\lambda,n)|$ is odd.
\end{rem}

\begin{rem}\label{RS}
Soon after the submission for the first version of this paper to the arXiv, Travis Scrimshaw kindly suggested that the odd property for non-skew case can be shown by considering a specialization $G_\lambda(\beta,\beta,\dots,\beta ; {-\beta^{-1}})$~\cite{[Sc]}.
Namely, the equality $G_\lambda(\beta,\beta,\dots,\beta ; {-\beta^{-1}})=\beta^{|\lambda|}$ can be proved in the following three ways.
\begin{itemize}
\item  The Grothendieck polynomial $G_{\lambda}$ is written by the wavefunction of the five-vertex model (see Lemma 5.2 in~\cite{[MS13]}).
This leads to the equality.
This is interesting because it can be extended to $G_{\lambda\slash\!\!\slash\mu}$.
\item It is obtained by the bi-alternant formula~(\ref{W}). 
We perform these calculations explicitly in Appendix A (see Theorem~\ref{AA} and Remark~\ref{AR}).
\item  By using an application of a sign-reversing involution, we have the equality.
This idea is due to Darij Grinberg~\cite{[Sc]}.
\end{itemize}
He also suggested that a sign-reversing involution might be extendable to a more general setting. 
Following his suggestion, we will give a more elegant proof of the Theorem~\ref{M} in Appendix~\ref{B.1}.
\end{rem}
$\\$

\vspace{-10mm}
\appendix
\section*{Appendices}\label{[AA]}

\renewcommand{\appendixname}{Appendix}

\vspace{-2mm}
\section{Proof by using the bi-alternant formula}\label{A.1}
We provide another proof of the odd property of the number of set-valued tableaux $|{\rm SVT}(\lambda,n)|$ by using formulas of the {\it stable Grothendieck polynomial} $G_\lambda$. 
The polynomial $G_\lambda$ has tableaux-sum formula and the bi-alternant formula.
We only use these two formulas for the proof of the Theorem \ref{AA}.

For $T \in {\rm SVT}(\lambda,n)$, the weight and  the monomial of $x=(x_1,x_2,\dots,x_n) \in \mathbb{C}^n$ of $T$ are defined~(\ref{wei}) and ~(\ref{mono}), respectively.
The Grothendieck polynomial $G_\lambda$ has the tableaux-sum formula as follows~\cite{[Buc02]}:
\begin{equation} \label{G}
G_\lambda(x ; \beta)=G_\lambda = \sum_{T \in {\rm SVT}(\lambda)}\beta^{|T|-|\lambda|}x^{\omega(T)}.
\end{equation}
Setting $\beta = 0$ in~(\ref{G}), $G_\lambda$ coincides with the Schur polynomial $s_\lambda$.
 Indeed, we find that:
\begin{align*}
G_\lambda(x ; 0) &= \sum_{\substack{T\in{\rm SVT}(\lambda)\\|T|=|\lambda|}}x^{\omega(T)}
=\sum_{T \in {\rm SST}(\lambda)}x^{\omega(T)}
= s_\lambda(x).
\end{align*}
The Schur polynomial $s_\lambda$ has the bi-alternant formula:
\begin{align*}
s_\lambda = \frac{|x_i^{\lambda_i+n-j}|_{n \times n}}{\prod_{1\leq i < j \leq n}(x_i - x_j)}.
\end{align*}
Analogously, the Grothendieck polynomial $G_\lambda$ has the bi-alternant formula as follows~\cite{[IN13]} (see also~\cite{[Len00]}):
\begin{equation} \label{W}
G_\lambda = \frac{|x_i^{\lambda_i+n-j}(1+\beta x_i)^{j-1}|_{n \times n}}{\prod_{1\leq i < j \leq n}(x_i - x_j)}.
\end{equation}
Note that if it needs to add zeros to $\lambda$ such as
\begin{equation*}
\lambda = (\underbrace{\lambda_1,\lambda_2,\dots,\lambda_{\ell(\lambda)},0,\dots,0}_{n})\eqqcolon (\lambda_1,\dots,\lambda_n)
\end{equation*} 
for applying~(\ref{W}) to $G_\lambda$.
We can also immediately find that by setting $\beta=0$ for~(\ref{W}), we get $s_\lambda$.

\begin{case}
Let $\lambda = (2,1) \vdash 3$, $n=3$ and $x=(x_1,x_2,x_3)$. 
Taking a $T \in {\rm SVT}((2,1),3)$ as in Example~\ref{ex2.1}.\\
\begin{equation*}
{\raisebox{-5.5pt}[0pt][0pt]{$T$\ =\ }} {\raisebox{-2.5pt}[0pt][0pt]{\ytableaushort{{1}{12},{23}}}},
\end{equation*} \\ \\
We compute the weight $\omega(T)$ and the corresponding monomial $x^{\omega(T)}$,
\begin{align*}
\omega(T) &= (2,2,1), \\
x^{\omega(T)} &= x_1^2x_2^2x_3.
\end{align*}
By considering all tableaux in ${\rm SVT}((2,1),3)$ and their corresponding monomials,
the equation~(\ref{G}) then yields:
\begin{align} \label{y_1}
G_{(2,1)}(x ; \beta) =& s_{(2,1)} + \beta(x_1^2x_2^2 + x_1^2x_3^2 + x_2^2x_3^2 + 3x_1^2x_2x_3 + 3x_1x_2^2x_3 + 3x_1x_2x_3^2) \notag \\ 
&+ \beta^2(2x_1^2x_2^2x_3 + 2x_1^2x_2x_3^2 + 2x_1x_2^2x_3^2) + \beta^3(x_1^2x_2^2x_3^2).
\end{align}
Conversely, applying~(\ref{W}) to $\lambda=(2,1,0)$ gives:
\begin{align} \label{y_2}
G_{(2,1)}(x ; \beta) &= \frac{\begin{vmatrix}
x_1^{4} & x_1^{2}(1+\beta x_1) & (1 +\beta x_1)^{2}\\
x_2^{4} & x_2^{2}(1+\beta x_2) & (1 +\beta x_2)^{2}\\
x_3^{4} & x_3^{2}(1+\beta x_2) & (1 +\beta x_3)^{2}
\end{vmatrix}}{(x_1-x_2)(x_1-x_3)(x_2-x_3)}.
\end{align}
Calculating the right-hand side of~(\ref{y_2}), we confirm its coincidence with the right-hand side of (\ref{y_1}).
\end{case}
The following result is established using~(\ref{W}).
\begin{thm}\label{AA}
For any $\lambda \vdash l$, we have
\begin{align}\label{A1}
G_{\lambda}(\beta,\beta,\dots,\beta ; {-\beta^{-1}})=\beta^{|\lambda|}.
\end{align}
\end{thm}
\begin{proof}
Let $p_\lambda(x;\beta)$ denote the numerator of the right-hand side of~(\ref{W}) and $\Delta_n$ its denominator:\vspace{-2mm}
\begin{align*}
p_\lambda(x ; \beta)=|x_i^{\lambda_i+n-j}(1+\beta x_i)^{j-1}|_{n \times n},\quad \Delta_n=\prod_{1\leq i < j \leq n}(x_i - x_j).
\end{align*}\vspace{-1mm}
Consider the specialization $G_{\lambda}(x_1,x_2,\dots,x_{n-1},\beta ; {-\beta^{-1}})$. Evaluating the numerator yields\vspace{-2mm}
\begin{align} \label{a_5}
&p_\lambda(x_1,\dots,x_{n-1},\beta ; {-\beta^{-1}}) \notag \\
&= \begin{vmatrix}
x_1^{\lambda_1 + n -1} & x_1^{\lambda_2 + n-2}(1-{\beta^{-1}}x_1) & \cdots & x_1^{\lambda_k}(1-{\beta^{-1}}x_1)^{n-1}\\
\vdots & \vdots & \ & \vdots \notag \\
x_{n-1}^{\lambda_1 + n -1} & x_{n-1}^{\lambda_2 + n-2}(1-{\beta^{-1}}x_{n-1}) & \cdots & x_{n-1}^{\lambda_k}(1-{\beta^{-1}}x_{n-1})^{n-1}\\
\beta^{\lambda_1+n-1} & \beta^{\lambda_2+n-2}(1-{\beta^{-1}}\beta) &\cdots & (1-{\beta^{-1}}\beta)^{n-1}
\end{vmatrix} \notag \\[3mm]
&= \begin{vmatrix}
x_1^{\lambda_1 + n -1} & x_1^{\lambda_2 + n-2}(1-{\beta^{-1}}x_1) & \cdots & x_1^{\lambda_k}(1-{\beta^{-1}}x_1)^{n-1}\\
\vdots & \vdots & \ & \vdots \notag \\
x_{n-1}^{\lambda_1 + n -1} & x_{n-1}^{\lambda_2 + n-2}(1-{\beta^{-1}}x_{n-1}) & \cdots & x_{n-1}^{\lambda_k}(1-{\beta^{-1}}x_{n-1})^{n-1}\\
\beta^{\lambda_1+n-1} & 0 &\cdots & 0
\end{vmatrix} \notag \\[3mm]
&=(-1)^{n-1}\beta^{\lambda_1+n-1} \!(1-{\beta^{-1}}x_1)\!\cdots \!(1-{\beta^{-1}}x_{n-1})p_{\widehat{\lambda}}(\tilde{x} ; -\beta^{-1}) \notag \\[3mm]
&=\beta^{\lambda_1} \!(x_1-\beta)\!\cdots \!(x_{n-1}-\beta)p_{\widehat{\lambda}}(\tilde{x} ; -\beta^{-1} ),
\end{align}
where $\widehat{\lambda}=(\lambda_2,\dots,\lambda_{n}),\ \tilde{x}=(x_1,\dots,x_{n-1})$ and
\begin{align*}
p_{\widehat{\lambda}}(\tilde{x} ; -\beta^{-1})=\begin{vmatrix}
x_1^{\lambda_2 + n -2} & \cdots & x_1^{\lambda_k}(1 - \beta^{-1} x_1)^{n-2}  \\
\vdots & \ & \vdots \\
x_{n-1}^{\lambda_2 + n -2} & \cdots & x_{n-1}^{\lambda_k}(1 - \beta^{-1}x_{n-1})^{n-2} 
\end{vmatrix}. \notag
\end{align*}
For the denominator, substituting $x_n = \beta$ gives\vspace{-2mm}
\begin{align} \label{a_6}
\Delta_n(x_1,x_2,\dots,\beta)=\Delta_{n-1}(\tilde{x})\cdot(x_1-\beta)\cdots(x_{n-1}-\beta).
\end{align}
Combining~(\ref{a_5}) and~(\ref{a_6}) leads to\vspace{-2mm}
\begin{align}
G_{\lambda}(x_1,x_2,\dots,x_{n-1},\beta ; {-\beta^{-1}})&=\frac{p_\lambda(x_1,\dots,x_{n-1},\beta ; {-\beta^{-1}})}{\Delta_n(x_1,x_2,\dots,\beta)}\notag\\
&=\beta^{\lambda_1}\frac{p_{\widehat{\lambda}}(\tilde{x} ; {-\beta^{-1}})}{\Delta_{n-1}(\tilde{x})} \notag \\
&=\beta^{\lambda_1}G_{\widehat{\lambda}}(\tilde{x} ; {-\beta^{-1}}) \notag.
\end{align}
Specializing $x_{n-1},x_{n-2},\dots,x_1$ in this order and performing the same calculation, this yields
\begin{align*}
G_{\lambda}(\beta,\beta,\dots,\beta ; {-\beta^{-1}})=\beta^{\lambda_1+\lambda_2+\cdots+\lambda_n}G_{\emptyset}=\beta^{|\lambda|}.
\end{align*}
This finishes the proof of Theorem~\ref{AA}.
\end{proof}
We obtain the following results from Theorem~\ref{AA}.
\begin{cor} \label{AC}
For any $\lambda \vdash l$, the number $|{\rm SVT}(\lambda,n)|$ is odd.
\end{cor}
\begin{proof}
Substituting $\beta=1$ into~(\ref{A1}) yields
\begin{align*}
G_{\lambda}(1,1,\dots,1 ; {-1})=1.
\end{align*}
Combining this with~(\ref{SV}) gives
\begin{align*}
|{\rm SVT}(\lambda,n)|&=G_{\lambda}(1,1,\dots,1 ; 1)\\
&\equiv G_{\lambda}(1,1,\dots,1 ; {-1})\pmod 2\\
&=1.
\end{align*}
Thus, the result is proved.
\end{proof}

\begin{rem}\label{AR}
We showed Corollary~\ref{AC} by calculating $G_{\lambda}(1,1,\dots,1 ; 1)$ in our previous preprint.
By considering $G_{\lambda}(1,1,\dots,1 ; -1)$, a more refined proof is achieved. This is due to the comments~\cite{[Sc]}.
\end{rem}

\vspace{-5mm}
\section{A sign-reversing involution}\label{B.1}
In this appendix, we provide another proof of Theorem~\ref{M} by constructing an involution on the set of set-valued tableaux.
We give a formula for a special value of the skew Grothendieck polynomial $G_\theta$ as an application of this involution.
\begin{proof}[Another proof of Theorem~\ref{M}]
Let $T_0 \in {\rm SVT}(\theta,n)$ be the unique tableau defined by the following condition:
\begin{align*}
\sum_{(i,j)\in \theta}\sum_{m \in (T_0)_{i,j}}m = \min_{T \in {\rm SVT}(\theta,n)}\sum_{(i,j) \in \theta}\sum_{m \in T_{i,j}}m.
\end{align*}
This tableau is constructed by assigning $\{1\}, \{2\}, \{3\}, \dots$ from top to bottom in each column.
We set $\mathrm{SVT}'(\theta,n)=\mathrm{SVT}(\theta,n)\setminus{\{T_0\}}$.
We define an order on the distinct boxes $b_{i,j},b_{i',j'} \in \theta $ in the following:
\begin{align*}
b_{i,j} < b_{i',j'} \quad\overset{\text{def}}{\Longleftrightarrow}\quad
\begin{cases}
    j < j', \\
    j = j' \ \text{and} \ i < i'.
  \end{cases}
\end{align*}
For a given tableau $T \in {\rm SVT}'(\theta,n)$, we identify the minimal box $b_{i,j}$ with respect to this order where $T_{i,j} \ne (T_0)_{i,j}$. We set $(T_0)_{i,j} = \{k\}\ (k \in [n])$.

The map $\iota : {\rm SVT}'(\theta,n) \to {\rm SVT}'(\theta,n)$ is defined as follows:\vspace{-1mm}
\begin{align*}
&\iota : {\rm SVT}'(\theta,n) \rightarrow {\rm SVT}'(\theta,n) \\
  &T_{i,j} \mapsto
  \begin{cases}
    T_{i,j} \setminus \{k\} & (k \in T_{i,j}),\\
    T_{i,j} \sqcup \{k\} & (k \notin T_{i,j}),
  \end{cases}\\
   &T_{i',j'} \mapsto T_{i',j'}\ \  (b_{i',j'}\neq b_{i,j}).
\end{align*}

As an example, consider the skew Young diagram $\theta = (5,3,3,2)/(3,1,1)$ with $n = 3$. The tableau $T_0 \in {\rm SVT}(\theta,3)$ is
\vspace{2mm}
\begin{align*}
{\raisebox{-5.5pt}[0pt][0pt]{$T_0$\ =\ }} {\raisebox{1pt}[0pt][0pt]{\ytableaushort{\none\none\none11,\none11,\none22,13}}.\ }
\end{align*}$\\\\\\\\$
The map $\iota$ acts on two tableaux as follows:
\vspace{2mm}
\begin{align*}
T_1 &= \raisebox{1pt}{\ytableaushort{\none\none\none1{123},\none11,\none2{{\color{red}{2}}3},13}} \quad \longrightarrow \quad \iota(T_1) = \raisebox{1pt}{\ytableaushort{\none\none\none1{123},\none11,\none2{3},13}}, \\[5mm]
T_2 &= \raisebox{1pt}{\ytableaushort{\none\none\none{\color{blue}{2}}{23},\none11,\none2{2},13}} \quad \longrightarrow \quad \iota(T_2) = \raisebox{1pt}{\ytableaushort{\none\none\none{{\color{red}{1}}{\color{blue}{2}}}{23},\none11,\none2{2},13}}.
\end{align*}

To confirm that $\iota(T) \in {\rm SVT}'(\theta,n)$, we verify that the following conditions are satisfied:
\begin{align}
&\text{If } b_{i,j-1} \in \theta, \text{ then } \max(\iota(T)_{i,j-1}) \leq \min(\iota(T)_{i,j}); \label{z1} \\
&\text{If } b_{i-1,j} \in \theta, \text{ then } \max(\iota(T)_{i-1,j}) < \min(\iota(T)_{i,j}); \label{z2} \\
&\text{If } b_{i,j+1} \in \theta, \text{ then } \max(\iota(T)_{i,j}) \leq \min(\iota(T)_{i,j+1}); \label{z3} \\
&\text{If } b_{i+1,j} \in \theta, \text{ then } \max(\iota(T)_{i,j}) < \min(\iota(T)_{i+1,j}). \label{z4}
\end{align}
Conditions~\eqref{z1} and~\eqref{z2} follow from the minimality of $b_{i,j}$. 
The remaining conditions~\eqref{z3} and~\eqref{z4} are preserved by the definition of the map $\iota$.
It is straightforward to verify that the map $\iota$ is an involution: $\iota^2 = \mathrm{id}$ and $\iota(T) \ne T$ for all $T \in {\rm SVT}'(\theta,n)$. 
It follows that $|{\rm SVT}'(\theta,n)|$ is even, and hence the number $|{\rm SVT}(\theta,n)|$ is odd.
\end{proof}

Using the involution $\iota$ defined above, we obtain a special value of the skew Grothendieck polynomial $G_{\theta}$ as well.
Given a tableau $T \in {\rm SVT}(\theta,n)$, define its weight by
\begin{align}\label{wei}
\omega(T) = (\omega_1(T),\omega_2(T),\dots,\omega_n(T)) \in \mathbb{Z}_{\geq 0}^n,
\end{align}
where $\omega_m(T) \coloneqq |\{ m \in [n] \mid m \in T \}|$.
For a variable vector $x = (x_1,x_2,\dots,x_n) \in \mathbb{C}^n$, define the monomial associated with $T$ by
\begin{align}\label{mono}
x^{\omega(T)} = x_1^{\omega_1(T)}x_2^{\omega_2(T)}\cdots x_n^{\omega_n(T)}.
\end{align}
The skew Grothendieck polynomial $G_{\theta}$ is defined as follows:
\begin{align}\label{GG}
G_{\theta}(x_1,\dots,x_n; \beta) = \sum_{T \in {\rm SVT}(\theta,n)} \beta^{|T|-|\theta|} x^{\omega(T)},
\end{align}
where $\beta$ is a parameter and $|T|$ denotes the total number of entries in $T$.

Applying the involution $\iota$ to the sum in~\eqref{GG}, we obtain the following result:

\begin{thm}\label{B}
For any skew shape \( \theta = \lambda / \mu \), we have
\begin{align}\label{skew}
G_{\theta}(\beta,\dots,\beta; -\beta^{-1}) = \beta^{|\theta|}.
\end{align}
\end{thm}

\begin{proof}
Since the $\iota$ maps ${\rm SVT}'(\theta,n)$ to itself and satisfies $|\iota(T)| = |T| \pm 1$, we have
\begin{align*}
\sum_{T \in {\rm SVT}'(\theta,n)} (-\beta^{-1})^{|T|-|\theta|} \beta^{|T|}
&= \sum_{T \in {\rm SVT}'(\theta,n)} (-\beta^{-1})^{|\iota(T)| - |\theta|} \beta^{|\iota(T)|} \\
&= -\sum_{T \in {\rm SVT}'(\theta,n)} (-\beta^{-1})^{|T| - |\theta|} \beta^{|T|}.
\end{align*}
This cancellation yields
\begin{align*}
G_{\theta}(\beta,\dots,\beta; -\beta^{-1}) = \beta^{|T_0|} + \sum_{T \in {\rm SVT}'(\theta,n)} (-\beta^{-1})^{|T|-|\theta|} \beta^{|T|} = \beta^{|T_0|}+0 = \beta^{|\theta|}.
\end{align*}
We remark that $|T_0|=|\theta|$ since $T_0 \in {\rm SST}(\theta,n)$.
This completes the proof.
\end{proof}

\begin{rem}
Since the map $\iota$ satisfies $\iota^2={\rm id}$ and $(-1)^{|\iota(T)|}=-(-1)^{|T|}$, this map is called \textit{sign-reversing involution}.
\end{rem}

\begin{rem}
In~\cite{[HJKSS24],[HJKSS25]}, more general cases corresponding to Theorem~\ref{B} are discussed.
Note, however, that the identity~\eqref{skew} is not stated explicitly there.
\end{rem}

\begin{rem}
The \textit{shifted set-valued tableaux} were introduced in~\cite{[IN13]} to provide a combinatorial definition of the $K$-theoretic Schur $P$- and $Q$-functions. 
As one type of generalization, a skew version of these functions was introduced in~\cite{[LM21]}, where they are defined using \textit{shifted set-valued skew tableaux}. 
In~\cite{[NS24]}, we prove the same result for the $K$-theoretic Schur $P$- and $Q$-functions:
\begin{thm}[\cite{[NS24]}]
Let $\lambda$ and $\mu$ be strict partitions such that $\lambda \supset \mu$.
We have
\begin{align*}
GP_{\lambda/\mu}(\beta,\dots,\beta ; -\beta^{-1})=\beta^{|\lambda/\mu|},\\
GQ_{\lambda/\mu}(\beta,\dots,\beta ; -\beta^{-1})=\beta^{|\lambda/\mu|}.
\end{align*}
\end{thm}
These special values are obtained by applying a similar argument that constructs sign-reversing involutions on the sets of shifted set-valued skew tableaux, ${\rm SSVT}_P(\lambda/\mu,n)$ and ${\rm SSVT}_Q(\lambda/\mu,n)$. 
The uniqueness of $T_0$, and the well-definedness and involutiveness of $\iota$ can be verified by analogy with the detailed proofs presented in~\cite{[NS24]}.
\end{rem}

\section{A simple proof for the oddness}\label{C.1}
In this appendix we give a simple proof for Theorem~\ref{M}\footnote{
	This appendix is based on suggestions from the referee.
	The authors would like to thank the referee for providing valuable comments.}.
\begin{proof}[Another proof of Theorem~\ref{M}]
	We define a map $\varphi:\mathrm{SVT}(\lambda/\mu,n)\to\mathrm{SST}(\lambda/\mu,n)$ as follows:
	\begin{align*}
		\varphi(T)_{i,j}=\max T_{i,j}\quad (T\in\mathrm{SVT}(\lambda/\mu,n),\ b_{i,j}\in\lambda/\mu).
	\end{align*}
	By the semistandard condition of the set-valued tableaux, we easily find that $\varphi(T)\in \mathrm{SST}(\lambda/\mu,n)$.
	We consider the inverse image $\varphi^{-1}(T)$ for each $T\in\mathrm{SST}(\lambda/\mu,n)$.
	For any $T\in\mathrm{SST}(\lambda/\mu,n)$, we have
	\begin{align*}
		\varphi^{-1}(T)=\{T'\in\mathrm{SVT}(\lambda/\mu,n)\mid \max T'_{i,j}=T_{i,j}\},
	\end{align*}
	due to the definition of $\varphi$.
	On the other hand, every $T'\in\mathrm{SVT}(\lambda/\mu,n)$ satisfies that
	\begin{itemize}
		\item $\max T'_{i,j-1}\leq \min T'_{i,j}$ if $b_{i,j-1},b_{i,j}\in\lambda/\mu$,
		\item $\max T'_{i-1,j}<\min T'_{i,j}$ if $b_{i-1,j},b_{i,j}\in\lambda/\mu$,
	\end{itemize}
	by the semistandard condition for $\mathrm{SVT}(\lambda/\mu,n)$.
	For $T'\in \varphi^{-1}(T)$, we have $\max T'_{i,j-1}=T_{i,j-1}$ when $b_{i,j-1}\in\lambda/\mu$ and $\max T'_{i-1,j}=T_{i-1,j}$ when $b_{i-1,j}\in\lambda/\mu$.
	We put $c(T)_{i,j}=\max (T_{i,j-1},T_{i-1,j}+1)$ for $T\in\mathrm{SST}(\lambda/\mu,n)$.
	Here we promise that $T_{i,j}=0$ if $b_{i,j}\notin \lambda/\mu$.
	Then we find that
	\begin{align}
		\varphi^{-1}(T)
		&=\{T'\mid c(T)_{i,j}\leq \min T'_{i,j},\ \max T'_{i,j}\leq T_{i,j}\}
		\notag\\
		&=\{T'\mid T'_{i,j}\subset \{c(T)_{i,j},c(T)_{i,j}+1,\ldots,T_{i,j}\}\}.\label{setpullback}
	\end{align}
	When we construct an element of the set \eqref{setpullback}, there are $2^{T_{i,j}-c(T)_{i,j}+1}$ ways to insert a subset of $[n]$ into the box $b_{i,j}$.
	Therefore $|\varphi^{-1}(T)|$ is a power of $2$, in particular, $|\varphi^{-1}(T)|$ is $1$ or an even number.
	
	We prove that $|\varphi^{-1}(T)|=1$ if and only if $T=T_0$ in the following.
	Here $T_0$ is given in Appendix~\ref{B.1}.
	The tableau $T_0$ is originally defined as an element of $\mathrm{SVT}(\lambda/\mu,n)$.
	Each $(T_0)_{i,j}$ is a singleton, so we identify $T_0$ as an element of $\mathrm{SST}(\lambda/\mu,n)$.
	If $T=T_0$, we easily find $c(T)_{i,j}=T_{i,j}$ for every $b_{i,j}$.
	Thus we have $|\varphi^{-1}(T)|=1$.
	On the other hand, if $c(T)_{i,j}=T_{i,j}$ for every $b_{i,j}$, we find that $T$ is constructed by assigning $1, 2, 3, \dots$ from top to bottom in each column.
	Therefore we have $T=T_0$.
	This means that $|\varphi^{-1}(T)|=1$ if and only if $T=T_0$.
	
	Since
	\begin{align*}
		\mathrm{SVT}(\lambda/\mu,n)=\bigsqcup_{T\in\mathrm{SST}(\lambda/\mu,n)}\varphi^{-1}(T)=\varphi^{-1}(T_0)\sqcup\bigsqcup_{T\neq T_0}\varphi^{-1}(T),
	\end{align*}
	we obtain
	\begin{align*}
		|\mathrm{SVT}(\lambda/\mu,n)|
		&=|\varphi^{-1}(T_0)|+\sum_{T\neq T_0}|\varphi^{-1}(T)|
		\notag\\
		&=1+\sum_{T\neq T_0}|\varphi^{-1}(T)|
		\notag\\
		&\equiv 1\quad (\mathrm{mod}\ 2).
	\end{align*}
	This completes the proof of Theorem~\ref{M}.
\end{proof}
Theorem~\ref{B} also can be derived by using the map $\varphi$.
We have
\begin{align*}
	&G_\theta(\beta,\ldots,\beta;-\beta^{-1})
	\\
	&=\sum_{T'\in\mathrm{SVT}(\theta,n)}(-\beta^{-1})^{|T'|-|\theta|}\beta^{|T'|}
	\\
	&=\sum_{T\in\mathrm{SST}(\theta,n)}\sum_{T'\in\varphi^{-1}(T)}(-\beta^{-1})^{|T'|-|\theta|}\beta^{|T'|}
	\\
	&=\beta^{|\theta|}+(-\beta)^{|\theta|}\sum_{T\neq T_0\in\mathrm{SST}(\theta,n)}\sum_{T'\in\varphi^{-1}(T)}(-1)^{|T'|}.
\end{align*}
We consider the term $\!\displaystyle \sum_{T'\in\varphi^{-1}(T)}\!(-1)^{|T'|}$ for $T\neq T_0\in\mathrm{SST}(\theta,n)$.
We recall \eqref{setpullback}:
\begin{align*}
	\varphi^{-1}(T)&=\{T'\mid T'_{i,j}\subset \{c(T)_{i,j},c(T)_{i,j}+1,\ldots,T_{i,j}\}\}\\
	&=\prod_{b_{i,j}\in\theta}\{T'_{i,j}\subset \{c(T)_{i,j},c(T)_{i,j}+1,\ldots,T_{i,j}\}\}.
\end{align*}
Using this equality, we have
\begin{align*}
	\sum_{T'\in\varphi^{-1}(T)}(-1)^{|T'|}
	&=\prod_{b_{i,j}\in\theta} \sum_{T'_{i,j}\subset \{c(T)_{i,j},c(T)_{i,j}+1,\ldots,T_{i,j}\}}(-1)^{|T'_{i,j}|}
	\\
	&=\prod_{b_{i,j}\in\theta} (-1)^{c(T)_{i,j}}(1-1)^{T_{i,j}-c(T)_{i,j}}.
\end{align*}
Since $T\neq T_0$, there is a box $b_{i,j}$ such that $T_{i,j}-c(T)_{i,j}>0$.
Therefore we obtain
\begin{align*}
	\sum_{T'\in\varphi^{-1}(T)}(-1)^{|T'|}=0
\end{align*}
for $T\neq T_0$.
This completes the proof of Theorem~\ref{B}.

\section*{Acknowledgments}

The authors would like to express their sincere gratitude to Yasuhiko Yamada for valuable discussions and constant encouragement throughout this work. 
They are also deeply indebted to Takeshi Ikeda for insightful comments and helpful advice. 
Special thanks are due to Travis Scrimshaw for carefully reading the manuscript and offering many constructive suggestions that greatly improved the presentation. 
This work was supported by JST SPRING, Grant Number JPMJSP2148.

\noindent{\sc Department of Mathematics, Graduate School of Science, Kobe University}

(Taikei Fujii) {\it E-mail address}: {\tt tfujii@math.kobe-u.ac.jp}

\vspace{3mm}

\noindent{\sc Department of Education, Kogakkan University.}

(Takahiko Nobukawa) {\it E-mail address}:  {\tt t-nobukawa@kogakkan-u.ac.jp}

\vspace{3mm}

\noindent{\sc Department of Mathematics, Graduate School of Science, Kobe University}

(Tatsushi Shimazaki) {\it E-mail address}: {\tt tsimazak@math.kobe-u.ac.jp}

\end{document}